\documentclass[a4paper,10pt]{amsart}
\usepackage[colorlinks,linkcolor=blue,citecolor=blue]{hyperref}
\usepackage{latexsym, amssymb, amsmath, amscd, amsthm, mathrsfs, bbm}
\usepackage[all, knot]{xy}
\xyoption{arc}

\usepackage{anysize}\marginsize{25mm}{25mm}{30mm}{30mm}
\allowdisplaybreaks[4]

\def \k {\mathbbm{k}}

\def \V {\operatorname{V}_{n,d}}
\def \U {\operatorname{U}_{n,d}}

\def \D {\operatorname{Diag}_{n,d}}
\def \dim {\operatorname{dim}}
\def \rk {\operatorname{Rank}}

\def \GL {\operatorname{GL}}
\def \Im {\operatorname{Im}}
\def \Q {\mathbbm{Q}}
\def \R {\mathbbm{R}}
\def \C {\mathbbm{C}}

\numberwithin{equation}{section}
\numberwithin{table}{section}
\numberwithin{equation}{section}
\newtheorem{theorem}{Theorem}[section]
\newtheorem{lemma}[theorem]{Lemma}
\newtheorem{proposition}[theorem]{Proposition}
\newtheorem{corollary}[theorem]{Corollary}
\newtheorem{definition}[theorem]{Definition}
\newtheorem{example}[theorem]{Example}
\newtheorem{remark}[theorem]{Remark}
\newtheorem{algorithm}[theorem]{Algorithm}

\title{Diagonalizable Higher Degree Forms and Symmetric Tensors}
\thanks{Supported by NSFC 11971181 and 11971449.}



\author{Hua-Lin Huang, Huajun Lu, Yu Ye and Chi Zhang}

\address{School of Mathematical Sciences, Huaqiao University, Quanzhou 362021, China}
\email{hualin.huang@hqu.edu.cn}

\address{School of Mathematical Sciences, Huaqiao University, Quanzhou 362021, China}
\email{huajun@hqu.edu.cn}

\address{School of Mathematical Sciences, University of Science and Technology of China, Hefei 230026, China}
\email{yeyu@ustc.edu.cn}

\address{Department of Mathematics, Northeastern University, Shenyang 110819, China}
\email{zhangchi@mail.neu.edu.cn}

\begin{document}

\begin{abstract}
We provide simple criteria and algorithms for expressing homogeneous polynomials as sums of powers of independent linear forms, or equivalently, for decomposing symmetric tensors into sums of rank-1 symmetric tensors of linearly independent vectors. The criteria rely on two facets of higher degree forms, namely Harrison's algebraic theory and some algebro-geometric properties. The proposed algorithms are elementary and based purely on solving linear and quadratic equations. Moreover, as a byproduct of our criteria and algorithms one can easily decide whether or not a homogeneous polynomial or symmetric tensor is orthogonally or unitarily decomposable.
\end{abstract}

\maketitle

\section{Introduction}
The paper is concerned about the problem: given a homogeneous polynomial $f \in \C[x_1, \cdots, x_n]$ of degree $d \ge 3,$ decide whether it is equivalent to $x_1^d + \cdots + x_n^d$ (called diagonalizable) and if so, give efficient algorithms to express the equivalence. It has an equivalent version in terms of symmetric tensors: given a symmetric $d$-tensor $A=(a_{i_1 \cdots i_d})_{1 \le i_1, \cdots, i_d \le n},$ provide criteria and algorithms to express $A=v_1^{\otimes d} + \cdots + v_n^{\otimes d}$ where the $v_i$'s are a basis of $\C^n.$ Though this is a typical problem of classical invariant theory (especially in the polynomial version), it has aroused much interest in the communities of applied mathematics and computing sciences. As a matter of fact, the problem has been explicitly proposed in some previous papers \cite{cm, ka, ks} of signal processing and computational complexity. In addition, assuming a homogeneous polynomial is equivalent to a sum of powers of independent linear forms, people are interested in whether one can choose the change of variables to be orthogonal or unitary. See for example \cite{kolda, ro, bdhr, ko} among many works in this direction.

In this paper we tackle the problem (mainly in the polynomial version) through two important aspects of higher degree forms (another popular synonym of homogeneous polynomials of degree at least $3$): the algebraic theory, a higher analogue of Witt's algebraic theory of quadratic forms, initiated by Harrison \cite{h1, h2}, and algebro-geometric properties of sums of powers. The crux is Harrison's centers of higher degree forms, which can be seen as a generalization of symmetric matrices. In particular, the direct sum decompositions of a form is in bijection with the decompositions of the unit of its center algebra into orthogonal idempotents. On the other hand, it is obvious that diagonalizable forms enjoy nice algebro-geometric properties. Specifically, they are smooth and this imposes very strong restriction, namely the semisimplicity, on their centers. With this one can easily show that a form $f \in \C[x_1, \cdots, x_n]$ is diagonalizable if and only if its center algebra $Z(f) \cong \C \times \cdots \times \C$ ($n$ copies). This enables us to provide several criteria and algorithms for diagonalizing higher degree forms or symmetric tensors. We present in detail a simple algorithm involving only linear equations (for computing the center) and quadratic equations (for computing orthogonal primitive elements and presenting the explicit diagonalization). Furthermore, with a help of Harrison's uniqueness result of the decomposition of a form into a direct sum of indecomposable forms \cite{h1}, we notice that the orthogonal or unitary decomposability of a form (over appropriate ground fields) is a simple property of its diagonalizability. As long as the diagonalization is determined, one can decide whether a form is orthogonally or unitarily decomposable by a straightforward check on any chosen change of variables for the diagonalization.

It is natural to extend the main strategy of the present paper to two more general problems. One is the Waring decompositions of forms, that is decomposing forms into sums of powers of linear forms which are linearly dependent in general. In terms of symmetric tensors, this is the decompositions into symmetric rank-1 terms. A possible way is to perturb a given nondiagonalizable form into a diagonalizable one in more variables. The other is the direct sum decompositions of forms, or equivalently the decompositions of symmetric tensors into sums of block terms. This is reduced to a thorough understanding of the semsimple quotient of the center algebra of any given form. These problems will be addressed in our forthcoming work.

The remainder of this paper is organized as follows. In Section 2 we recall the basic notions and mutual interpretations of higher degree forms, symmetric tensors and symmetric multilinear spaces. Section 3 is devoted to centers and direct sum decompositions of forms. The main results, criteria and algorithms for (orthogonal or unitary) diagonalizations of higher degree forms, are presented in Section 4. In Section 5 we provide some examples to elucidate the criteria and algorithms. Throughout the paper, let $d \ge 3$ be an integer and we consider forms and symmetric tensors of degree $d.$ Although our motivating problem is over the complex numbers $\C,$ in most cases we can work over a general field $\k$ with $\operatorname{char} \k = 0,$ or $\operatorname{char} \k > d.$

\section{Higher degree forms, symmetric tensors and symmetric multilinear spaces}
Higher degree forms are homogeneous polynomials of degree $d \ge 3.$ Similar to the familiar situation of quadratic forms, higher degree forms are naturally associated to symmetric tensors (i.e., symmetric multi-dimensional matrices) and to symmetric multilinear spaces.

Let $f(x_1, \cdots, x_n) \in \k[x_1, \cdots, x_n]$ be a form of degree $d.$ For convenience, we write $f$ in the symmetric way:
\[
f(x_1, \cdots, x_n) = \sum_{1 \le i_1, \cdots, i_d \le n} a_{i_1 \cdots i_d} x_{i_1} \cdots x_{i_d}
\]
where the $a_{i_1 \cdots i_d}$'s are symmetric with respect to their indices.
The resulting symmetric $d$-tensor $A=(a_{i_1 \cdots i_d})_{1 \le i_1, \cdots, i_d \le n}$ is called the Gram tensor of $f.$  We also write the form $f(x_1, \cdots, x_n)=A x^d$ in terms of products of tensors, where $x=(x_1, \cdots, x_n)^T$ is the vector of variables. Corresponding to the form $f$ there is also an associated symmetric $d$-linear space. Let $V$ be a vector space over $\k$ of dimension $n$ with a basis $\alpha_1, \cdots, \alpha_n.$ Define $\Theta \colon V \times \cdots \times V \longrightarrow \k$ by \[\Theta(\alpha_{i_1}, \cdots, \alpha_{i_d})=a_{i_1 \cdots i_d}, \quad \forall 1 \le i_1, \cdots, i_d \le n.\] The pair $(V, \Theta)$ is called the associated symmetric $d$-linear space of $f$ under the basis $\alpha_1, \cdots, \alpha_n.$ One can recover the form $f$ from $(V, \Theta)$ as\[ f(x_1, \cdots, x_n) = \Theta\left(\sum_{1 \le i \le n}x_i\alpha_i, \ \dots, \ \sum_{1 \le i \le n}x_i\alpha_i\right).\]

If $x=Py$ with $P=(p_{ij}) \in \GL(n,\k)$ a change of variables, then the resulting form is \[ g(y_1, \cdots, y_n)=\sum_{1 \le j_1, \cdots, j_d \le n} \ \sum_{1 \le i_1, \cdots, i_d \le n} a_{i_1 \cdots i_d}p_{i_1j_1} \cdots p_{i_dj_d}y_{j_1} \cdots y_{j_d} \] and the associated Gram tensor becomes \[ AP^d:= \left( \sum_{1 \le i_1, \cdots, i_d \le n} a_{i_1 \cdots i_d}p_{i_1j_1} \cdots p_{i_dj_d} \right)_{1 \le j_1, \cdots, j_d \le n}.\] Call $AP^d$ the $d$-congruence of $A$ by $P.$ Let $(\beta_1, \cdots, \beta_n)=(\alpha_1, \cdots, \alpha_n)P,$ then under this new basis the associated symmetric $d$-linear space reads
\[ \Theta(\beta_{j_1}, \cdots, \beta_{j_d}) = \sum_{1 \le i_1, \cdots, i_d \le n} a_{i_1 \cdots i_d}p_{i_1j_1} \cdots p_{i_dj_d}, \quad \forall 1 \le j_1, \cdots, j_d \le n.\]

The form $f$ is called diagonalizable over $\k$ if
\begin{equation*}
f(x_1, \cdots, x_n) = \sum_{1 \le i \le r} \lambda_i l_i(x_1, \cdots, x_n)^d,
\end{equation*}
where the $\lambda_i$'s are nonzero constants in $\k$ and the $l_i$'s are independent $\k$-linear forms. Clearly, if this is the case then $r \le n.$ Accordingly, the symmetric tensor $A$ is called diagonalizable over $\k$ if there exists a $P \in \GL(n,\k)$ such that $AP^d$ is diagonal (i.e., the entries are $0$ unless the indices are identical), and the symmetric $d$-linear space $(V, \Theta)$ is called diagonalizable if there exists a basis $\beta_1, \cdots, \beta_n$ of $V$ such that $\Theta(\beta_{j_1}, \cdots, \beta_{j_d}) = 0$ unless $j_1=\cdots=j_d.$ Moreover, if $l_i(x_1, \cdots, x_n)=v_{i1}x_1 + \cdots + v_{in}x_n \ (1 \le i \le r)$ and denote $v_i=(v_{i1}, \cdots, v_{in}),$ then the corresponding decomposition of the associated symmetric tensor is
\[ A= \lambda_1 v_1^{\otimes d} + \cdots + \lambda_r v_r^{\otimes d}. \]
A diagonalizable form is called orthogonally or unitarily diagonalizable if there is an orthogonal or unitary change of variables for the diagonalization. Of course, in this case one needs to work on appropriate ground fields so that orthogonal and unitary groups are well-defined. We also remark that one may replace $r$ by $n$ in the previous definition, as one can consider only nondegenerate forms without loss of generality. See Section 3 for more explicit explanation.

In the following we don't distinguish the three synonyms, namely higher degree forms, symmetric tensors and symmetric multilinear spaces. The discussions are mostly presented in terms of higher degree forms. One can easily shift to the versions of symmetric tensors and symmetric multilinear spaces.

\section{Centers and direct sum decompositions}
In his pioneering work \cite{h1} of algebraic theory of higher degree forms, Harrison introduced the notion of centers (in terms of symmetric multilinear spaces) to deal with the direct sum decompositions of forms.

\begin{definition}
Let $(V,\Theta)$ be a symmetric $d$-linear space. The center, denoted by $Z(V, \Theta),$ of $(V,\Theta)$ is defined as \[ \{ \phi \in \operatorname{End}(V) \mid \Theta(\phi(v_1), v_2, \cdots, v_d)=\Theta(v_1, \phi(v_2), \cdots, v_d), \ \forall v_1, v_2, \cdots, v_d \in V \}. \]
\end{definition}

Let $f$ be the associated degree $d$ form of $(V, \Theta)$ under a basis $\alpha_1, \cdots, \alpha_n.$ By $H$ we denote the Hessian matrix $(\frac{\partial^2 f}{\partial x_i \partial x_j})_{1 \le i,\ j \le n}$ of the form $f$ and by $A^{(i_3 \cdots i_d)}$ the $n \times n$ matrix $(a_{i_1i_2i_3 \cdots i_d})_{1 \le i_1, i_2 \le n}$ where $A=(a_{i_1 \cdots i_d})_{1 \le i_1, \cdots, i_d \le n}$ is the associated symmetric $d$-tensor of $f.$ Then we have the following equivalent definitions of centers in terms of forms (see \cite{h2}) and tensors.

\begin{lemma} \label{edc} Keep the above notations. Then we have
\[Z(V, \Theta) \cong \{ X \in \k^{n \times n} \mid (HX)^T =  HX \} = \{ X \in \k^{n \times n} \mid X^TA^{(i_3 \cdots i_d)}=A^{(i_3 \cdots i_d)}X, \ \forall 1 \le i_3, \cdots, i_d \le n \}. \]
\end{lemma}

\begin{remark} \label{cb}
The definition of centers by $(V,\Theta)$ is coordinate free, while the other versions by Hessian matrices and slices of tensors are not. Let $P \in \GL(n,\k)$ be a change of coordinates. If we write $$Z(A)=\{ X \in \k^{n \times n} \mid X^TA^{(i_3 \cdots i_d)}=A^{(i_3 \cdots i_d)}X, \ \forall 1 \le i_3, \cdots, i_d \le n \},$$ then $Z(AP^d)=P^{-1}Z(A)P=\{P^{-1}XP \mid X \in Z(A)\}.$

Indeed, assume $P=(p_{ij})$ and denote $B:=AP^d.$ Notice that
\begin{eqnarray*}
  B^{(j_3 \cdots j_d)} &=& \left( \sum_{1 \le i_1, \cdots, i_d \le n} a_{i_1 \cdots i_d}p_{i_1j_1}p_{i_2j_2}p_{i_3j_3} \cdots p_{i_dj_d} \right)_{1 \le j_1, j_2 \le n} \\
   &=& P^T\left( \sum_{1 \le i_3, \cdots, i_d \le n} a_{i_1 \cdots i_d}p_{i_3j_3} \cdots p_{i_dj_d} \right)_{1 \le i_1, i_2 \le n}P \\
   &=& P^T\left( \sum_{1 \le i_3, \cdots, i_d \le n} p_{i_3j_3} \cdots p_{i_dj_d} A^{(i_3 \cdots i_d)} \right)P.
\end{eqnarray*}
Denote $C^{(j_3 \cdots j_d)}=p_{i_3j_3} \cdots p_{i_dj_d} A^{(i_3 \cdots i_d)}$ and $P^{-1}=(q_{ij}).$ Then we have
\[ A^{(i_3 \cdots i_d)}=\sum_{1 \le j_3, \cdots, j_d \le n} q_{j_3 i_3 } \cdots q_{j_d i_d } C^{(j_3 \cdots j_d)}. \]
Now it follows that
\begin{eqnarray*}
  Z(B) &=& \{ Y \in \k^{n \times n} \mid Y^TB^{(j_3 \cdots j_d)}=B^{(j_3 \cdots j_d)}Y, \ \forall 1 \le j_3, \cdots, j_d \le n \} \\
   &=& \{ Y \in \k^{n \times n} \mid Y^T P^T A^{(i_3 \cdots i_d)}P=P^TA^{(i_3 \cdots i_d)}PY, \ \forall 1 \le i_3, \cdots, i_d \le n \} \\
   &=& \{ Y \in \k^{n \times n} \mid PYP^{-1} \in Z(A) \}= P^{-1}Z(A)P.
\end{eqnarray*}
\end{remark}

Now we recall the definition of direct sum decompositions of forms.
\begin{definition}
A form $f$ is called a direct sum if, after a change of variables, it can be written as a sum of $t \ge 2$ nonzero forms in disjoint sets of variables:
\begin{equation*}\label{ds}
f=f_1(x_1, \cdots, x_{a_1}) + \cdots + f_t(x_{a_{t-1}+1}, \cdots, x_n).
\end{equation*}
If this is not the case, then $f$ is said to be indecomposable. On the other extreme, $f$ is diagonalizable if the $f_i$'s are forms in only one variable.
\end{definition}

It is clear that direct sum decompositions of $f$ correspond to quasi-diagonalizations of its Gram tensor by $d$-congruence, and to orthogonal decompositions of its associated symmetric $d$-linear space (see \cite{h1}).

In this paper we are mainly interested in diagonalizable forms. For this there is no loss of generality in assuming the forms are nondegenerate as in \cite{h1}, that is no variable can be removed by an invertible linear change of variables. In other words, a form $f \in \k[x_1, \cdots, x_n]$ is degenerate if there exists a change of variables $x=Py$ such that the resulting form $g$ involves less than $n$ variables. Or equivalently, in terms of symmetric multilinear spaces, there exists $0 \ne u \in V$ such that $\Theta(u, v_2, \cdots, v_d)=0$ for all $v_2, \cdots, v_d \in V.$ For the associated symmetric $d$-tensor $A,$ let $A_{i_1}$ denote the $(d-1)$-tensor $A=(a_{i_1 \cdots i_d})_{1 \le i_2, \cdots, i_d \le n}.$ Then $f$ is degenerate if the $A_{i_1}$'s are linearly dependent in the space of $(d-1)$-tensors. Similar to the quadratic case, define the radical of the symmetric $d$-linear space by \[\ker \Theta := \{ u \in V \mid \Theta(u, v_2, \cdots, v_d)=0, \quad \forall v_2, \cdots, v_d \in V \}\] and the slicing rank of the symmetric $d$-tensor by \[ \rk A := \rk \{ A_1, \cdots, A_n \}, \]
namely the size of any maximally linearly independent subset of $\{ A_1, \cdots, A_n \}.$ Note that $\rk A$ is independent of the ways of slicing along the indices as $A$ is symmetric. The rank of $f,$ denoted by $\rk f,$ is defined to be the essential number $r$ of variables of $f,$ that is, there exists an invertible change of variables $x=Py$ such that the resulting form $g$ involves exactly $r$ variables, say $g=g(y_1, \cdots, y_r)$ and is nondegenerate in $\k[y_1, \cdots, y_r].$

\begin{lemma} \label{rof}
Keep the previous notations. Then $$\rk f = \rk A=n-\dim \ker \Theta.$$
\end{lemma}

\begin{remark}
One can easily give a proof mimicking the quadratic situation. As pointed out in \cite{h1}, any form can be decomposed as the direct sum of a nondegenerate form and a zero form. The slicing rank is Kruskal's 1-slabs rank \cite{kr} for symmetric tensors. The essential number of variables of forms was also considered in \cite{c} via Apolarity Theory and Catalecticant Matrices.
\end{remark}

Now we recall some important facts about centers and direct sum decompositions of forms which are useful in this paper. See \cite[Propositions 2.3, 4.1 and 4.3]{h1} for proofs.

\begin{proposition} \label{pc}
Suppose $f \in \k[x_1, \cdots, x_n]$ is a nondegenerate higher degree form. Then
\begin{itemize}
\item[(1)] The center $Z(f)$ is a commutative subalgebra of $\operatorname{End}(V).$
\item[(2)] If $f=f_1 + \cdots + f_t$ is a direct sum decomposition, then $Z(f) \cong Z(f_1) \times \cdots \times Z(f_t).$
\item[(3)] If $1=\epsilon_1+\cdots+\epsilon_t$ is a decomposition of orthogonal idempotents for the unit of $Z(f),$ then $f=f_1 + \cdots + f_t$ is a direct sum decomposition with the corresponding orthogonal decomposition of symmetric $d$-linear space $V=\Im \epsilon_1 \oplus \cdots \oplus \Im \epsilon_t.$
\item[(4)] $f$ is indecomposable over $\k$ if and only if $Z(f)$ is a local $\k$-algebra.
\item[(5)] $f$ is diagonalizable over $\k$ if and only if $Z(f) \cong \k \times \cdots \times \k$ ($n$ copies).
\item[(6)] The decomposition of $f$ into a direct sum of indecomposable forms is unique up to equivalence and permutation of indecomposable summands.
\item[(7)] If $K/\k$ is a field extension, by $f_K$ it is meant treating $f \in K[x_1, \cdots, x_n],$ then $Z(f_K) \cong Z(f) \otimes_\k K.$
\end{itemize}
\end{proposition}

For later applications, we present some more explanations in particular for item (3). Let $(V, \Theta)$ be the associated symmetric $d$-linear space of $f$ under the basis $\alpha_1, \cdots, \alpha_n.$ For the given idempotent decomposition $1=\epsilon_1+\cdots+\epsilon_t,$ assume $\dim \Im \epsilon_i = n_i$ and choose a basis $\beta_{i1}, \cdots, \beta_{in_i}$ for $\Im \epsilon_i.$ Then $\beta_{11}, \cdots, \beta_{1n_1}, \cdots, \beta_{t1}, \cdots, \beta_{tn_t}$ make a basis of $V$ and there exists a $P=(p_{ij}) \in \GL(n,\k)$ such that $$(\alpha_1, \cdots, \alpha_n)=(\beta_{11}, \cdots, \beta_{1n_1}, \cdots, \beta_{t1}, \cdots, \beta_{tn_t})P.$$ For convenience, simplify the indices of the new basis as $\beta_1, \cdots, \beta_n.$ Then we have
\begin{eqnarray*}
  f(x_1, \cdots, x_n) &=& \Theta\left(\sum_{1 \le i \le n}x_i\alpha_i, \ \dots, \ \sum_{1 \le i \le n}x_i\alpha_i\right) \\
   &=& \Theta\left(\sum_{1 \le j \le n} \sum_{1 \le i \le n}p_{ji}x_i\beta_j, \ \dots, \ \sum_{1 \le j \le n} \sum_{1 \le i \le n}p_{ji}x_i\beta_j\right) \\
   &=& \Theta\left(\sum_{j=1}^{n_1} \sum_{1 \le i \le n}p_{ji}x_i\beta_j, \ \dots, \ \sum_{j=1}^{n_1} \sum_{1 \le i \le n}p_{ji}x_i\beta_j\right) \\
   & & + \cdots +\\
   & & \Theta\left(\sum_{j=n-n_t+1}^{n} \sum_{1 \le i \le n}p_{ji}x_i\beta_j, \ \dots, \ \sum_{j=n-n_t+1}^{n} \sum_{1 \le i \le n}p_{ji}x_i\beta_j\right).
\end{eqnarray*}
Note that the last equality is due to the orthogonality of the $\epsilon_i$'s. Let $y=Px$ be the change of variables and denote
\begin{eqnarray*}
  g_1(y_1, \cdots, y_{n_1}) &=& \Theta\left(\sum_{j=1}^{n_1} y_j\beta_j, \ \dots, \ \sum_{j=1}^{n_1} y_j\beta_j\right), \\
   &\vdots&  \\
  g_t(y_{n-n_t+1}, \cdots, y_{n}) &=& \Theta\left(\sum_{j=n-n_t+1}^{n} y_j\beta_j, \ \dots, \ \sum_{j=n-n_t+1}^{n} y_j\beta_j\right).
\end{eqnarray*}
Now we have the desired direct sum decomposition $$f=g_1(y_1, \cdots, y_{n_1})+\cdots+g_t(y_{n-n_t+1}, \cdots, y_{n}).$$

It is well known that all quadratic forms over fields with $2 \ne 0$ are diagonalizable. On the contrary, as an easy consequence of Lemma \ref{edc} and Proposition \ref{pc}, we have

\begin{proposition} \label{g}
A general higher degree form is central and absolutely indecomposable in the sense of Harrison \cite{h1}. That is, its center is the ground field and it remains indecomposable over any extension of the ground field.
\end{proposition}

\begin{proof}
A general $n$-variate degree $d$ form can be written as
\[ f(x_1, \cdots, x_n) = \sum_{|I|=d} a_Ix^I.\]
In this notation $I=(i_1, \cdots, i_n)$ is a multi-index, $0 \le i_k \le n$ for each $k,$ ranging through ${n+d-1}\choose{d}$ values for which $|I|:=\sum_{1 \le k \le n} i_k=d.$ For each multi-index, the monomial $x^I=x_1^{i_1} \cdots x_n^{i_n}$ comes with coefficient $a_I \in \k.$ Assume $C=(c_{ij}) \in Z(f).$ Then by Lemma \ref{edc} we have
\[ \sum_{1 \le k \le n} \frac{\partial^2 f}{\partial x_i \partial x_k}c_{kj} = \sum_{1 \le k \le n} \frac{\partial^2 f}{\partial x_j \partial x_k}c_{ki}, \quad \forall i \ne j. \]
As the form $f$ is general (so one may view the $a_I$'s as indeterminates), for any fixing pair $i \ne j$ the union \[ \left\{ \frac{\partial^2 f}{\partial x_i \partial x_k} \right\}_{1 \le k \le n} \ \bigcup \ \left\{ \frac{\partial^2 f}{\partial x_j \partial x_k} \right\}_{1 \le k \le n}\] is a linearly independent set. So the previous equation induces $c_{ii}=c_{jj}$ and $c_{ij}=c_{ji}=0$ for all $i \ne j.$ That means a general form $f$ is central in the sense of Harrison \cite{h1}, i.e. $Z(f) \cong \k.$ It follows by items (4) and (7) of Proposition \ref{pc} that $f$ is absolutely indecomposable.
\end{proof}

\section{Main results}
In this section, take the ground field $\k =\C,$ the case of our interest. In order to obtain criteria and algorithms for diagonalizable forms we put their algebro-geometric properties, in particular the obvious smoothness, into consideration.

Recall that a form $f \in \C[x_1, \cdots, x_n]$ is called smooth, if the simultaneous equations \[\frac{\partial f }{ \partial x_1}= \cdots = \frac{\partial f }{ \partial x_n} =0\] have no nonzero solutions. In terms of the associated symmetric multilinear space $(V,\Theta),$ this is equivalent to saying that
\[ \Theta(u, \cdots, u, v_d)=0, \ \forall v_d \in V \ \Rightarrow \ u=0. \]
It is clear that, if $f$ is smooth, then it is nondegenerate.

It was observed by Harrison \cite{h1} that the smoothness of $f$ implies that its center $Z(f)$ has no nontrivial nilpotent elements. In the present situation, we have

\begin{lemma} \label{ls}
 Suppose $f \in \C[x_1, \cdots, x_n]$ is a smooth form of degree $d.$ Then $Z(f) \cong \C \times \cdots \times \C.$ Consequently, $\dim Z(f)$ is the number of indecomposable summands of $f$ and $\dim Z(f) \le n.$ In particular, $f$ is diagonalizable if and only if $\dim Z(f) = n.$
\end{lemma}

\begin{proof}
For the sake of completeness, we recall Harrison's argument that $Z(f)$ has no nontrivial nilpotent elements. Assume the contrary, there is a nilpotent element $\phi \in Z(f)$ with $\phi^{m+1}=0$ while $\phi^m \ne 0$ for some $m >1.$ Then there is some $v \in V$ such that $ \phi^m(v) \ne 0.$ Hence $$\Theta(\phi^m(v), \cdots, \phi^m(v), v_d)=\Theta(\phi^{(d-1)m}(v), \cdots, v, v_d)=0$$ for all $v_d \in V$ as $(d-1)m >m +1.$ Now the smoothness condition of $f$ forces $\phi^m(v)=0.$ This is absurd.

The rest is easy. Since $Z(f)$ is commutative with zero radical, by the Wedderburn-Artin theorem of semisimple algebras, $Z(f) \cong \C \times \cdots \times \C.$ Now by Proposition \ref{pc}, indecomposable direct summands of $f$ are in bijection with complete sets of orthogonal primitive idempotents of $Z(f).$ Hence the number of indecomposable direct summands of $f$ is exactly $\dim Z(f).$ As is obvious that this number can not exceed the number of variables of $f,$ hence $\dim Z(f) \le n.$

Finally, if $f$ is diagonalizable, then its number of indecomposable summands is $n,$ thus $\dim Z(f) = n.$ Conversely, if $\dim Z(f) = n,$ then $Z(f) \cong \C \times \cdots \times \C$ ($n$ copies) since $Z(f)$ is a semisimple commutative $\C$-algebra. It follows by item (5) of Proposition \ref{pc} that $f$ is diagonalizable.
\end{proof}

From now on, let $\V \subset \C[x_1, \cdots, x_n]$ denote the linear space of forms of degree $d$ in $n$ variables. Let $\U \subset \V$ be the set of smooth forms. It is well known that $\U$ is an open subset of $\V$ defined by an irreducible polynomial $\Delta_{n,d} \ne 0,$ where $\Delta_{n,d}$ is the discriminant  (see Chapter 13 in \cite{gkz}). Let $\D \subset \U$ denote the subset of diagonalizable forms.

In view of Proposition \ref{g}, diagonalizable higher degree forms should be very special and rare. This is explicitly confirmed in the following

\begin{proposition}\label{cd}
$\D$ is a proper closed subset of $\U.$
\end{proposition}

\begin{proof}
Obviously, $\D$ is not empty. Thanks to Proposition \ref{g}, $\D$ is a proper subset of $\U.$ Note by Lemma \ref{edc} that the center $Z(f)$ of a form $f$ is the solution space of a system of linear equations, written simply as $Cy=0,$ which is obtained by combining together all the matrix equations therein. Consider the entries of $C$ as linear forms of the coefficients of a general form in $\U.$ Due to Lemma \ref{ls}, a form $f \in \U$ is diagonalizable if and only if $\dim Z(f) = n,$ thus if and only if $\rk C = n^2 -n.$ Again by Lemma \ref{ls}, we have $\rk C \ge n^2 -n$ since $n^2 -\rk C = \dim Z(f) \le n.$ Now it follows that $\D$ is a closed subset of $\U$ defined by all the $(n^2-n+1)$-minors of $C.$
\end{proof}

In summary, we have the following criteria for diagonalizable forms of higher degree. For convenience, we consider $Z(f) \subset \C^{n \times n}$ and let $D$ denote the subalgebra of $\C^{n \times n}$ consisting of all diagonal matrices.

\begin{theorem}\label{m}
Suppose $f \in \V$ is nondegenerate. Then the following statements are equivalent:
\begin{enumerate}
  \item The form $f$ is diagonalizable.
  \item The center $Z(f) \cong D.$
  \item $f$ is smooth and $\dim Z(f) =n.$
  \item $Z(f)$ is semisimple and $\dim Z(f) =n.$
  \item $\dim Z(f)=n$ and any basis of $Z(f)$ consists of diagonalizable matrices.
  \item $\dim Z(f)=n$ and $Z(f)$ has a basis consisting of rank 1 and trace 1 matrices.
\end{enumerate}
\end{theorem}

\begin{proof}
The equivalence of (1), (2), (3) and (4) is already contained in Lemma \ref{ls}. Clearly, (2) implies (5) and (6). Assume (5), then by a simultaneous diagonalization of any basis (since $Z(f)$ is commutative) one easily obtains $Z(f) \cong D.$ Finally assume (6). Note that a rank 1 and trace 1 matrix is obviously idempotent. The condition of (6) says exactly that $Z(f)$ has a basis consisting of $n$ commuting idempotent matrices. This immediately implies $Z(f) \cong D.$
\end{proof}

Now we turn to the problem of orthogonal or unitary diagonalizations of forms and symmetric tensors. This has been intensively studied in e.g. \cite{ro, bdhr, ko} and called symmetrically odeco (over $\R$) or udeco (over $\C$) therein. Apparently orthogonally or unitarily diagonalizable forms are \emph{a priori} diagonalizable. On the other hand, notice that the diagonalization (if exists) of a form is essentially unique due to the seminal work of Harrison \cite{h1}. So the key is the diagonalization, and the orthogonality or unitarity is a property which can be checked directly. Therefore, it is very easy to decide whether or not a higher degree form or symmetric tensor is orthogonally or unitarily diagonalizable based upon the previous theorem.

\begin{corollary} \label{ou}
Suppose $f \in \D$ and $P=(p_{ij})_{1 \le i, j \le n}$ is an invertible matrix such that \[ f= \sum_{1 \le i \le n} \alpha_i\left(\sum_{1 \le j \le n} p_{ij}x_j\right)^d.\] Then $f$ is orthogonally (resp. unitarily) diagonalizable if and only if, after appropriate scaling of each row, $P$ is orthogonal (resp. unitary).
\end{corollary}

\begin{proof}
Assume $f$ is orthogonally (resp. unitarily) diagonalizable, that is there exist an orthogonal (resp. unitary) matrix $Q=(q_{ij})_{1 \le i, j \le n}$ such that \[ f= \sum_{1 \le i \le n} \lambda_i\left(\sum_{1 \le j \le n} q_{ij}x_j\right)^d. \] By item (6) of Proposition \ref{pc}, there exists $\tau_i \in \C^*$ for all $1 \le i \le n$ such that \[ \sum_{1 \le j \le n} q_{\sigma(i)j}x_j = \tau_i \sum_{1 \le j \le n} p_{ij}x_j, \] where $\sigma$ is a permutation of $\{ 1, \dots, n\}.$ That is, $P$ is orthogonal (resp. unitary) up to scaling by a suitable diagonal matrix. The converse is trivial.
\end{proof}

As an example of our approach to symmetrically odeco and udeco, we give a simple proof for a main result in \cite{bdhr, ko}: the set of symmetrically odeco tensors can be described by equations of degree 2. We will adopt the version of characterization for degree 3 in \cite[Theorems 3.3 and 3.6]{ko}, but prove the result for all degrees. In the following we use the notations of Lemma \ref{edc} and we stick to the notions adopted in \cite{bdhr, ko} for coherence.

\begin{proposition}
A real (resp. complex) symmetric tensor $A=(a_{i_1 \cdots i_d})_{1 \le i_1, \cdots, i_d \le n}$ is odeco if and only if the $A^{(i_3 \cdots i_d)}$'s pairwise commute (resp.  and are diagonalizable).
\end{proposition}

\begin{proof}
We only prove the real case, as the complex case is similar. In addition, we may assume without loss of generality that $\rk A=n.$

Assume $A$ is odeco. Then there exists an orthogonal matrix $O$ such that $AO^d$ is diagonal. By Lemma \ref{edc}, it is immediate that $Z(AO^d)=\operatorname{D}_n(\R),$ the subalgebra of $\R^{n \times n}$ consisting of all diagonal matrices. It follows by Remark \ref{cb} that $Z(A)=O\operatorname{D}_n(\R)O^T.$ Again by Lemma \ref{edc}, $O^TA^{(i_3 \cdots i_d)}O$ commutes with $\operatorname{D}_n(\R)$ for all $1 \le i_3, \cdots, i_d \le n.$ Thus $O^TA^{(i_3 \cdots i_d)}O$ is diagonal for all $1 \le i_3, \cdots, i_d \le n.$ It follows right away that the $A^{(i_3 \cdots i_d)}$'s pairwise commute.

Conversely, assume the $A^{(i_3 \cdots i_d)}$'s pairwise commute. As they are all symmetric, there is an orthogonal matrix $O$ such that $O^TA^{(i_3 \cdots i_d)}O$ is diagonal for all $1 \le i_3, \cdots, i_d \le n.$ As $\rk A=n,$ the linear span of all the $O^TA^{(i_3 \cdots i_d)}O$ is $\operatorname{D}_n(\R).$ Then by Lemma \ref{edc}, we have $Z(A)=O\operatorname{D}_n(\R)O^T.$ Now It follows by items (3) and (5) of Proposition 3.6 that $A$ is odeco.
\end{proof}

The previous simple criteria provide several ways to detect and determine the (orthogonal, or unitary) diagonalizations of higher degree forms. For example, by (3) compute $\dim Z(f)$ (via the rank of the matrix $C$ in the proof of Proposition \ref{cd}) and detect the smoothness of $f$ (e.g., by the Jacobian criterion); by (4), compute $Z(f)$ and detect the semisimplicity of a basis (via the minimal polynomial of any basis element, using Euclid algorithm to detect whether the polynomial has multiple roots); by (5), compute $Z(f)$ and simultaneously diagonalize a basis; by (6), compute $Z(f)$ and find a complete set of orthogonal primitive idempotents from any chosen basis of $Z(f).$ We will focus on the fourth and give in detail a theoretical algorithm later on. The reason is two-fold: on the one hand, it detects and determines the explicit diagonalization simultaneously; on the other hand, it is relatively simple involving only linear and quadratic equations.

\begin{algorithm} \label{a}
Take any $f \in \V.$ Let $A$ be the associated symmetric $d$-tensor.
\begin{enumerate}
  \item[] Step 1: Detect the nondegeracy. \emph{Compute $\rk A$ according to Lemma \ref{rof}. If $\rk A=n,$ then continue; otherwise, say $\rk A=r<n,$ reduce the form $f$ into a nondegenerate one and restart in $\operatorname{V}_{r,d}$.}\\
  \item[] Step 2: Compute the center. \emph{Solve the linear equations as in Lemma \ref{edc}. If $\dim Z(f) < n,$ then $f$ is not diagonalizable and we stop; otherwise, choose a basis $P_1, \dots, P_n$ of $Z(f).$}\\
  \item[] Step 3: Compute the idempotents. \emph{Let $\lambda_1, \dots, \lambda_n$ be indeterminates. Consider the matrix $\epsilon=\sum_{1 \le i \le n} \lambda_i P_i$ and impose the conditions: $\rk \epsilon=1$ (i.e., all the $2 \times 2$-minors are zero) and $\operatorname{trace}(\epsilon)=1.$ Solve the equations. If there exists no solutions, then $f$ is not diagonalizable and we stop; otherwise choose an $\epsilon$ and continue.}\\
  \item[] Step 4: Diagonalize the form. \emph{Decompose $f$ according to the pair $(\epsilon, 1-\epsilon)$ of orthogonal idempotents as in Proposition \ref{pc}. Then we have $f=l_1(x_1, \dots, x_n)^d + g$ where $g$ is a nondegenerate form of degree $d$ in $n-1$ variables which are linear forms of $x_1, \dots, x_n$ independent from $l_1.$ Then return to step 1 and replace $f$ by $g \in \operatorname{V}_{n-1,d}.$ If the input $f$ is diagonalizable, then in $n-1$ steps the procedure stops and we end up with a diagonalization of $f.$}\\
  \item[] Step 5: Check orthogonality or unitarity. \emph{By direct computation, verify whether the resulting change of variables is orthogonal, or unitary, or neither.}
\end{enumerate}

\end{algorithm}

\begin{remark}
In principle, our algorithm can be easily adapted to diagonalize higher degree forms over none algebraically closed fields. Quadratic extensions of the ground field may be necessary to compute the idempotents. In comparison with some previous works \cite{bdhr, cm, ka, ko, ks, ro}, our algorithm seems more elementary. In particular, we don't have to compute eigenvalues and eigenvectors (which obviously involve solving higher degree algebraic equations) and no factorization of multivariate polynomials is needed.
\end{remark}

\section{Some examples}
In this section, we provide some examples which are solvable by hand. There are two main steps. The first is a routine calculation of linear equations according to Lemma \ref{edc} which gives centers as linear spaces. The second is to determine the algebraic structure of centers and single out idempotents to decompose the forms according to Steps 3 and 4 of Algorithm \ref{a}. Along the way, one can also see clearly how to decompose forms over none algebraically closed fields.

\begin{example}[Binary quartics]
Let $f_t=x_1^4+x_2^4+tx_1^2x_2^2 \in \k[x_1, \ x_2].$ Then it is easy to see that $f_t$ is smooth if and only if $t \neq \pm 2.$ Moreover, if $t=0,$ then clearly $f_t$ is diagonal. In the following suppose $t \neq 0, \ \pm 2.$ By an easy calculation we have $Z(f_t) \cong \k$ if $t \neq \pm 6,$ and $Z(f_{6}) \cong \k \times \k,$  $Z(f_{-6}) \cong \k[x]/(x^2+1).$ Therefore, if $t \neq \pm 6,$ then $f_t$ is absolutely indecomposable; $f_6= \frac{1}{2}[(x_1-x_2)^4 +(x_1+x_2)^4]$ is diagonalizable; if $\sqrt{-1} \in \k,$ then $f_{-6}= \frac{1}{2}[(x_1-\sqrt{-1} x_2)^4 +(x_1+\sqrt{-1} x_2)^4]$ is diagonalizable, otherwise $f_{-6}$ is indecomposable (but not absolutely). In addition, if $\sqrt{2} \in \k$ as well, then $f_6$ is orthogonally diagonalizable and $f_{-6}$ is unitarily diagonalizable.
\end{example}

\begin{example}[Ternary cubics]
Consider the normal form of nonsingular ternary cubics $f_\lambda=x_1^3+x_2^3+x_3^3+6\lambda x_1x_2x_3.$ Clearly $f_0$ is already diagonal. In the following suppose $\lambda \neq 0.$ By a straightforward computation we have $Z(f_\lambda)\cong\k$ if $\lambda^3 \neq 1,$ thus $f_\lambda$ is absolutely indecomposable in this case. If $\lambda^3=1$ but $\lambda \ne 1,$ then $Z(f_\lambda) \cong \k \times \k \times \k,$ consequently $f_\lambda=\frac{1}{3}[(\lambda x_1+x_2+x_3)^3+(x_1+\lambda x_2+x_3)^3+(x_1+x_2+\lambda x_3)^3]$ is diagonalizable (but neither orthogonal nor unitary). If $\lambda=1,$ then $Z(f_1) \cong \k \times \k[x]/(x^2+x+1).$ In this case, according to the explanation after Proposition \ref{pc} take the change of variables $x_1=y_1+y_2, \ x_2=y_1+y_3, \ x_3=y_1-y_2-y_3$ and we have the corresponding direct sum decomposition $f_1=9y_1^3-9(y_2^2y_3+y_2y_3^2).$ If $\k$ does not contain a primitive cubic root of unity, then $\k[x]/(x^2+x+1)$ is a field and thus $y_2^2y_3+y_2y_3^2$ is indecomposable. If $\omega \in \k$ is a primitive cubic root of unity, then $\k[x]/(x^2+x+1) \cong \k \times \k$ and we can further decompose $y_2^2y_3+y_2y_3^2.$ In this case we have the diagonalization of $f_1$ as
$$\frac{1}{3}[(x_1+x_2+x_3)^3+(x_1+\omega x_2 + \omega^2 x_3)^3+(x_1+\omega^2 x_2 + \omega x_3)^3].$$ Furthermore, if $\k$ also contains $\sqrt{3},$ then $f_1$ is unitarily diagonalizable.
\end{example}

\begin{example}
Consider the cubic form $$f= x_1^3 -3x_1^2x_2 + 3x_1x_2^2 + 3x_1^2x_3 + 3x_1x_3^2 - 6x_1x_2x_3 + 13x_2^3 - 3x_2^2x_3 - 9x_2x_3^2 + 15x_3^3$$ in $\Q[x_1, \ x_2, \ x_3].$ In accordance with Lemma \ref{edc}, let
\[ A^{(1)}=
\left(
\begin{array}{ccc}
 1 & -1  & 1  \\
 -1 &  1 & -1  \\
 1 & -1  &  1
\end{array}
\right), \quad
 A^{(2)}=
\left(
\begin{array}{ccc}
-1  & 1  & -1  \\
1  & 13  & -1  \\
 -1 & -1  & -3
\end{array}
\right), \quad
A^{(3)}=
\left(
\begin{array}{ccc}
1  & -1  & 1  \\
-1  &  -1 & -3  \\
 1 & -3 &  15
\end{array}
\right).
\]
Then by a direct calculation, $Z(f)=\{ X \in \Q^{3 \times 3} \mid A^{(i)}X=X^TA^{(i)}, \ 1 \le i \le 3 \}=\bigoplus_{1 \le i \le 3} \Q X^{(i)},$ where
\[ X^{(1)}=
\left(
\begin{array}{ccc}
 1 & -1 & 1  \\
 0 &  0 & 0  \\
 0 &  0 & 0
\end{array}
\right), \quad
 X^{(2)}=
\left(
\begin{array}{ccc}
0  & 1  & -1  \\
 0 &  1 & 0   \\
 0 &  0 & 1
\end{array}
\right), \quad
X^{(3)}=
\left(
\begin{array}{ccc}
0  & 1  & -5  \\
 0 &  0 & 1  \\
 0 &  -1 & 6
\end{array}
\right).
\]
Note that $X^{(1)}$ is of rank 1 and trace 1, hence it is a primitive idempotent. Then according to the pair $(X^{(1)}, I_3-X^{(1)})$ of idempotents we have the following decomposition
\[ (x_1 - x_2 +x_3)^3 +(14x_2^3-6x_2^2x_3-6x_2x_3^2+14x_3^3). \] Apply the same process to $g:=14x_2^3-6x_2^2x_3-6x_2x_3^2+14x_3^3.$ In fact, we can read from the $X^{(i)}$'s that $Z(g)=\Q\left(
                                                                                                   \begin{array}{cc}
                                                                                                     1 & 0 \\
                                                                                                     0 & 1 \\
                                                                                                   \end{array}
                                                                                                 \right)
\oplus \Q\left(
            \begin{array}{cc}
              0 & 1 \\
              -1 & 6 \\
            \end{array}
          \right)
\cong \Q[\sqrt{2}].$ This is a quadratic extension of $\Q.$ It follows immediately that $g$ is indecomposable over $\Q.$ Consequently $Z(f) \cong \Q \times \Q[\sqrt{2}]$ and $f$ is a direct sum but not diagonalizable over $\Q.$

However, over any field extension $K/\Q$ with $\sqrt{2} \in K,$ one has easily $Z(f) \otimes _\Q K \cong K \times K \times K$ and $f$ is diagonalizable over $K:$
\[ (x_1 - x_2 +x_3)^3 + [(1+\sqrt{2})x_2+(1-\sqrt{2})x_3)]^3 + [(1-\sqrt{2})x_2 + (1+\sqrt{2})x_3]^3 .\]
Finally by Corollary \ref{ou}, it is clear that any diagonalization of $f$ is neither orthogonal nor unitary.
\end{example}

\begin{example}
Consider the following rational quartic form
\begin{equation*}
\begin{aligned}
f=&x_1^4+4x_1^3x_2+4x_1^3x_3+4x_1^3x_4-12x_1^2x_2^2+12x_1^2x_2x_3-24x_1^2x_2x_4-12x_1^2x_3^2-24x_1^2x_3x_4+24x_1^2x_4^2\\
&+4x_1x_2^3-24x_1x_2^2x_3+12x_1x_2^2x_4-24x_1x_2x_3^2+96x_1x_2x_3x_4-24x_1x_2x_4^2+4x_1x_3^3+12x_1x_3^2x_4\\
&-24x_1x_3x_4^2+4x_1x_4^3+x_2^4+4x_2^3x_3+4x_2^3x_4+24x_2^2x_3^2-24x_2^2x_3x_4-12x_2^2x_4^2+4x_2x_3^3-24x_2x_3^2x_4\\
&+12x_2x_3x_4^2+4x_2x_4^3+x_3^4+4x_3^3x_4-12x_3^2x_4^2+4x_3x_4^3+x_4^4.
\end{aligned}
\end{equation*}
Similar to the previous example, let $A$ denote the Gram tensor of $f$ and set $A^{(i_3i_4)}$ as in Lemma \ref{edc} to compute the center to get $Z(f)=\bigoplus_{1 \le i \le 4} \Q X^{(i)},$ where
\begin{equation*}
\begin{aligned}
X^{(1)}=
\left(
\begin{array}{cccc}
 1 & 0  & 0 & 0\\
 0 & 1  & 0 & 0\\
 0 & 0  & 1 & 0\\
 0 & 0  & 0 & 1\\
\end{array}
\right), &\quad
X^{(2)}=
\left(
\begin{array}{cccc}
 0 & 1  & 0 & 0\\
 -1 & 1  & 0 & 0\\
 0 & 0  & 0 & 1\\
 0 & 0  & -1 & 1\\
\end{array}
\right), \\
X^{(3)}=
\left(
\begin{array}{cccc}
 0 & 0  & 1 & 0\\
 0 & 0  & 0 & 1\\
 -1 & 0  & 1 & 0\\
 0 & -1  & 0 & 1\\
\end{array}
\right), &\quad
X^{(4)}=
\left(
\begin{array}{cccc}
 0 & 0  & 0 & 1\\
 0 & 0  & -1 & 1\\
 0 & -1  & 0 & 1\\
 1 & -1  & -1 & 1\\
\end{array}
\right).
\end{aligned}
\end{equation*}
Denote $Y=\left(
            \begin{array}{cc}
              0 & 1 \\
              -1 & 1 \\
            \end{array}
          \right)$
and notice that $X^{(2)}=I_2 \otimes Y, \ X^{(3)}=Y \otimes I_2$ and $X^{(4)}=Y \otimes Y.$ Now it is ready to see that $$Z(f) \cong \Q[x]/(x^2-x+1) \otimes \Q[x]/(x^2-x+1) \cong \Q[\sqrt{-3}] \times \Q[\sqrt{-3}].$$ So $f$ is not diagonalizable over $\Q$ by Theorem \ref{m}, but is a direct sum of two rational indecomposable forms by item (3) of Proposition \ref{pc}.
In light of the structure of $Z(f),$ it is not hard to find out a pair of orthogonal idempotents:
\[ \epsilon_1 = \frac{1}{3}\left(
                             \begin{array}{cccc}
                               2 & -1 & -1 & 2 \\
                               1 & 1 & -2 & 1 \\
                               1 & -2 & 1 & 1 \\
                               2 & -1 & -1 & 2 \\
                             \end{array}
                           \right), \quad
  \epsilon_2 = \frac{1}{3}\left(
                             \begin{array}{cccc}
                               1 & 1 & 1 & -2 \\
                               -1 & 2 & 2 & -1 \\
                               -1 & 2 & 2 & -1 \\
                               -2 & 1 & 1 & 1 \\
                             \end{array}
                           \right). \]
Accordingly we take the following change of variables
\begin{gather*}
  x_1 = \frac{1}{3}(y_1+y_2+y_3), \quad x_2=\frac{1}{3}(y_1+y_3+y_4), \\
  x_3 = \frac{1}{3}(y_2+y_3+y_4), \quad x_4=\frac{1}{3}(y_1+y_2+y_4).
\end{gather*}
The corresponding direct sum decomposition is
$f=t(y_1,y_2)+t(y_3,y_4)$
with $$t(\alpha,\beta)=\frac{1}{9}(-\alpha^4+4\alpha^3\beta+12\alpha^2\beta^2+4\alpha\beta^3-\beta^4).$$

Furthermore over any field extension $K/\Q$ with $\sqrt{-3} \in K,$ the center $Z(f) \otimes_\Q K \cong K \times K \times K \times K$ and $f$ can be diagonalized as:
\[(x_1+\omega x_2+\omega^2x_3+x_4)^4+(\omega x_1+x_2+x_3+\omega^2x_4)^4+(\omega^2x_1+x_2+x_3+\omega x_4)^4+(x_1+\omega^2x_2+\omega x_3+x_4)^4 \] where $\omega=-\frac{1}{2}+\frac{\sqrt{-3}}{2}.$
Therefore any diagonalization of $f$ is neither orthogonal nor unitary.
\end{example}


\begin{thebibliography}{99}

\bibitem{bdhr}
A. Boralevi, J. Draisma, E. Horobet, E. Robeva, Orthogonal and unitary tensor decomposition from an algebraic perspective, Israel Journal of Mathematics 222 (2017) 223--260.

\bibitem{c}
E. Carlini, Reducing the number of variables of a polynomial, In Algebraic geometry and geometric modeling, Math. Vis., pages 237--247. Springer, Berlin, 2006.

\bibitem{cm}
P. Comon, B. Mourrain, Decomposition of quantics in sums of powers of linear forms, Signal Process. 53 (2) (1996) 93--107.


\bibitem{gkz}
I. M. Gelfand, M. M. Kapranov, A. V. Zelevinsky, Discriminants, Resultants and Multidimensional Determinants, Modern Birkh\"auser Classics, Birkh\"auser Boston, Boston, MA, 2008.

\bibitem{h1}
D. Harrison, A Grothendieck ring of higher degree forms, J. Algebra 35 (1975) 123--138.

\bibitem{h2}
D. Harrison, B. Pareigis, Witt rings of higher degree forms, Comm. Algebra 16 (6) (1988) 1275--1313.


\bibitem{ka}
N. Kayal, Efficient algorithms for some special cases of the polynomial equivalence problem, In Symposium on Discrete Algorithms (SODA), Society for Industrial and Applied Mathematics, January 2011.




\bibitem{ko}
P. Koiran, Orthogonal tensor decomposition and orbit closures from a linear algebraic perspective, Linear and Multilinear Algebra, 2019. arXiv:1905.05094.

\bibitem{ks}
P. Koiran, M. Skomra, Derandomization and absolute reconstruction for sums of powers of linear forms, 	arXiv:1912.02021.

\bibitem{kolda}
T. G. Kolda, Orthogonal tensor decompositions, SIAM J. Matrix Anal. Appl. 23 (2001) 243--255.

\bibitem{kr}
J. B. Kruskal, Three-way arrays: rank and uniqueness of trilinear decompositions, with application to arithmetic complexity and statistics, Linear Alg. Appl. 18 (1977) 95--138.








\bibitem{ro}
E. Robeva, Orthogonal Decomposition of Symmetric Tensors, SIAM J. Matrix Anal. Appl. 37 (2016) 86--102.





\end{thebibliography}
\end{document}